\documentclass[11pt, twoside]{article}

\usepackage{cite}
\usepackage{amssymb}
\usepackage{mathrsfs}
\usepackage{amsmath}
\usepackage{amsthm}
\usepackage{amsfonts}
\usepackage{latexsym}
\usepackage{indentfirst}
\usepackage{color}
\usepackage{enumerate}

\usepackage[english]{babel}

\usepackage{txfonts}
\usepackage{anysize}

\usepackage{float}
\usepackage{caption}
\usepackage[colorlinks=true,
linkcolor=blue,
citecolor=red,
urlcolor=magenta,
backref=page
]{hyperref}

\usepackage{tikz}
\usetikzlibrary{arrows,decorations.pathreplacing}
\definecolor{myred}{RGB}{255,0,0}
\definecolor{myblue}{RGB}{0,0,255}
\definecolor{mygreen}{RGB}{0,128,0}
\definecolor{myyellow}{RGB}{153, 0, 255}

\definecolor{cantorblue}{RGB}{35,93,157}
\definecolor{ancestororange}{RGB}{193,112,24}
\definecolor{localred}{RGB}{180,47,65}
\definecolor{descgreen}{RGB}{37,133,85}

\allowdisplaybreaks

\newtheorem{theorem}{Theorem}[section]
\newtheorem{lemma}[theorem]{Lemma}
\newtheorem{corollary}[theorem]{Corollary}
\newtheorem{proposition}[theorem]{Proposition}

\theoremstyle{definition}
\newtheorem{remark}[theorem]{Remark}

\numberwithin{equation}{section}

\begin{document}

\title{\bf\Large Sharp Matrix-Weighted Lower Bounds for
Non-Degenerate Convolution Calder\'on--Zygmund Operators
\footnotetext{\hspace{-0.35cm} 2020 {\it Mathematics Subject Classification}.
Primary 42B20; Secondary 47A30, 47A56, 42B35.\endgraf
{\it Key words and phrases.}
Calder\'on--Zygmund operator, matrix weight,
quantitative bound, Brown's lemma, Cantor set.\endgraf
This project is partially supported by
the National Natural Science Foundation of China
(Grant Nos. 12431006 and 12371093),
the Beijing Natural Science Foundation (Grant No. 1262011), and
the Fundamental Research Funds for the Central Universities (Grant Nos.
2253200028 and 312200502406).}}
\date{}
\author{Fan Bu, Dachun Yang\footnote{Corresponding author, E-mail:
\texttt{dcyang@bnu.edu.cn}/{\color{red} \today}/Final version.}
\ and Wen Yuan}

\maketitle

\vspace{-0.8cm}

\begin{center}
\begin{minipage}{13cm}
{\small {\bf Abstract}\quad
Let $T$ be a convolution Calder\'on--Zygmund operator satisfying
Hyt\"onen's non-degeneracy condition.
We prove that, for every $p\in(1,\infty)$ and $m\ge3$,
there exists a positive constant $C$ such that, for any $t\in[1,\infty)$,
\begin{equation*}
\sup_{[W]_{\mathcal A_p}\le t}
\|T\|_{L^p(W)\to L^p(W)}
\ge C t^{1+\frac{1}{p(p-1)}},
\end{equation*}
where the supremum is taken over all matrix weights
$W:\ \mathbb R^n\to\mathbb C^{m\times m}$ with their weight
constants $[W]_{\mathcal A_p}\le t$.
If, in addition, the odd part of the kernel is non-degenerate
(such as the Hilbert transform and the Riesz transforms),
the same conclusion holds for every $m\ge2$.
}
\end{minipage}
\end{center}

\vspace{0.1cm}

\tableofcontents

\vspace{0.1cm}

\section{Introduction}

Throughout this article, we work in $\mathbb{R}^n$ and, unless
otherwise specified, take $\mathbb{R}^n$ as the underlying space.

Weighted norm inequalities play a fundamental role in harmonic analysis
and provide essential tools for the study of singular integrals,
function spaces, and partial differential equations.
Let $p\in(1,\infty)$.
In 1972, Muckenhoupt \cite{Muckenhoupt} introduced the class of $A_p$ weights, defined by
\begin{equation}\label{eq:scalar-Ap}
[w]_{A_p}:=
\sup_{\mathrm{cube}\,Q\subset\mathbb R^n}
\fint_Q w
\left(\fint_Q w^{-\frac{1}{p-1}}\right)^{p-1}
<\infty,
\end{equation}
and proved that the $A_p$ condition characterizes
the boundedness of the Hardy--Littlewood maximal operator on $L^p(w)$.
Subsequently, Hunt et al. \cite{HMW} proved that
the Hilbert transform $H$ is bounded on $L^p(w)$ over $\mathbb R$
if and only if $w\in A_p$.
Once the qualitative theory had been established,
attention turned to the sharp dependence of operator norms
on the $A_p$ characteristic.
In 1993, Buckley \cite{Buckley} proved the sharp bound
for the Hardy--Littlewood maximal operator.
For Calder\'on--Zygmund operators, determining the optimal dependence of
the operator norm on the $A_p$ characteristic became a central problem,
motivated in part by its connection with regularity questions
for the Beltrami equation; see Astala et al. \cite{AIS}.
This problem is considerably more subtle.
The quantitative extrapolation theorem of
Dragi\v{c}evi\'c et al. \cite{DGPP} reduced the upper-bound problem
for general $p\in(1,\infty)$ to $p=2$.
Subsequently, Petermichl \cite{PetH,PetR} established the sharp weighted bounds
for the Hilbert transform and the Riesz transforms.
Hyt\"onen \cite{HytonenA2} eventually  proved the $A_2$ theorem
for general Calder\'on--Zygmund operators $T$, yielding
\begin{equation*}
\|T\|_{L^p(w)\to L^p(w)}
\lesssim [w]_{A_p}^{\max\{1,\frac1{p-1}\}}.
\end{equation*}
Subsequently, Lerner introduced the sparse domination approach \cite{l13b},
which yields a simpler proof of the $A_2$ theorem \cite{l13a}.
This approach has since profoundly influenced the study of
weighted inequalities (see, for example,
\cite{llo24,lor17,bd20,bqd22,bz26}).

To address problems arising in multivariate stationary processes
and Toeplitz operators, Treil and Volberg \cite{TV} introduced
matrix $\mathcal A_2$ weights and proved that
the matrix $\mathcal A_2$ condition characterizes the boundedness of
the Hilbert transform on $L^2(W)$ spaces over $\mathbb R$.
Their result was later extended independently to
$L^p(W)$ on $\mathbb R$ with $p\in(1,\infty)$ by Nazarov and Treil \cite{NT}
and by Volberg \cite{VolbergAp}, using different methods.
Christ and Goldberg \cite{ChristGoldberg} subsequently
established the $L^2$ boundedness of the matrix-weighted maximal operator,
while Goldberg \cite{Goldberg} extended this boundedness
to the full range $p\in(1,\infty)$ and, as an application,
proved that the matrix $\mathcal A_p$ condition characterizes
the boundedness on $L^p(W)$ of convolution Calder\'on--Zygmund operators
satisfying Stein's non-degeneracy condition.
For further qualitative results in the matrix-weighted setting,
see \cite{llsw,bctw,N2012,N2025,N2013,nr18}.

A matrix weight $W:\ \mathbb R^n\to\mathbb C^{m\times m}$
is a measurable, locally integrable function
taking values in the positive definite $m\times m$ matrices almost everywhere.
For a measurable vector-valued function
$\vec f:\ \mathbb R^n\to\mathbb C^m$, define
\begin{equation*}
\|\vec f\|_{L^p(W)}
:=\left[\int_{\mathbb R^n}\left|W^{\frac1p}(x)\vec f(x)\right|^p\,dx\right]^{\frac1p}.
\end{equation*}
Throughout this article, we use the following
equivalent definition of matrix $\mathcal A_p$ weights given by Roudenko \cite{Roudenko}:
\begin{equation}\label{eq:Ap}
[W]_{\mathcal A_p}:=
\sup_{\mathrm{cube}\,Q\subset\mathbb R^n}
\fint_Q \left[ \fint_Q \left\|W^{\frac1p}(x)W^{-\frac1p}(y)\right\|^{p'}\,dy
\right]^{\frac{p}{p'}}\,dx <\infty.
\end{equation}
Here, and thereafter, $\fint_Qh=|Q|^{-1}\int_Qh$ and $p':=\frac p{p-1}$.
For $m=1$, \eqref{eq:Ap} reduces exactly to \eqref{eq:scalar-Ap}.

On the quantitative upper-bound side,
an important starting point was the work of Bickel et al. \cite{BPW},
who obtained quantitative estimates for the Hilbert transform.
Subsequently, Nazarov et al. \cite{NPTV} introduced
convex body domination and proved that,
for general Calder\'on--Zygmund operators $T$,
\begin{equation*}
\|T\|_{L^2(W)\to L^2(W)}\lesssim [W]_{\mathcal A_2}^{\frac32},
\end{equation*}
thereby improving the earlier estimate in \cite{BPW}.
Around the same time, Culiuc et al. \cite[Corollary 5]{CDO}
obtained the same $[W]_{\mathcal A_2}^{\frac32}$ bound
through a related sparse domination approach based on
dyadic shifts and Hyt\"onen's representation theorem.
Subsequently, Cruz-Uribe et al. \cite[Corollary 1.16]{CIM}
extended this quantitative bound to the full range $p\in(1,\infty)$:
\begin{equation}\label{eq:cim-upper}
\|T\|_{L^p(W)\to L^p(W)}
\lesssim [W]_{\mathcal A_p}^{1+\frac1{p(p-1)}}.
\end{equation}
For more results on quantitative bounds in the matrix-weighted setting,
see \cite{bw16,DHL20,LLOR24,LLOR242,hpv2019,DLY,YYZ}.

Recently, Bownik and Cruz-Uribe \cite{BCU}
established matrix-weighted versions of
the Jones factorization theorem and
the sharp quantitative Rubio de Francia extrapolation theorem
by developing the theory of convex-set valued functions.
In particular, their extrapolation theorem shows that,
if the upper bound for Calder\'on--Zygmund operators could be
improved to depend linearly on $[W]_{\mathcal A_2}$,
then the sharp upper bounds known in the scalar-weighted setting
would also hold in the matrix-weighted setting for all $p\in(1,\infty)$.
However, such an improvement is impossible in general.
Indeed, in 2024, Domelevo et al. \cite[Theorem 1.1]{DPTV} obtained
a surprising lower bound of the Hilbert transform $H$:
for any $t\in[1,\infty)$,
\begin{equation*}
\sup_{[W]_{\mathcal A_2}\le t} \|H\|_{L^2(W)\to L^2(W)}
\gtrsim t^{\frac32}.
\end{equation*}
In 2026, Jiao et al. \cite[Theorem 1.1]{jqwx} extended this result to
the full range $p\in(1,\infty)$:
\begin{equation*}
\sup_{[W]_{\mathcal A_p}\le t} \|H\|_{L^p(W)\to L^p(W)}
\gtrsim t^{1+\frac1{p(p-1)}}.
\end{equation*}
Together with the upper bound \eqref{eq:cim-upper},
this establishes the sharp dependence on
the matrix $\mathcal A_p$ characteristic for the Hilbert transform
throughout the full range $p\in(1,\infty)$.
A \emph{natural question} is whether the exponent in \eqref{eq:cim-upper}
is also sharp for all non-degenerate convolution Calder\'on--Zygmund operators.

In this article, we affirmatively answer this question.
To state the main result, we begin by stating the assumptions
on the kernel $K$ of the non-degenerate convolution
Calder\'on--Zygmund operator $T$.
Throughout this article, $T$ is \emph{always assumed} to be of
convolution type, which is bounded on
$L^2$ and whose off-diagonal kernel
$K:\ \mathbb R^n\setminus\{\mathbf 0\}\to\mathbb C$ satisfies

\begin{enumerate}[{\rm(i)}]
\item the \emph{size condition}:
there exists a positive constant $C_K$ such that,
for any $x\in\mathbb R^n\setminus\{\mathbf 0\}$,
\begin{equation}\label{eq:size}
|K(x)|\le C_K |x|^{-n},
\end{equation}

\item the \emph{H\"older regularity condition}:
there exist positive constants $H_K$ and $\alpha\in(0,1]$ such that,
for any $x\in\mathbb R^n\setminus\{\mathbf 0\}$ and
$h\in\mathbb R^n$ with $|h|\le\frac{|x|}{2}$,
\begin{equation}\label{eq:holder}
|K(x-h)-K(x)|\le H_K \frac{|h|^\alpha}{|x|^{n+\alpha}},
\end{equation}

\item the \emph{non-degeneracy condition}:
there exists a positive constant $a$ such that,
for any $r\in(0,\infty)$,
\begin{equation}\label{eq:nd}
|K(x)|\ge a r^{-n}
\quad\text{ for some } x\in \mathbb R^n\setminus B(\mathbf 0,r).
\end{equation}
\end{enumerate}

Conditions \eqref{eq:size} and \eqref{eq:holder} are the convolution version
of the standard Calder\'on--Zygmund kernel conditions.
Condition \eqref{eq:nd} is exactly the convolution version of
Hyt\"onen's non-degeneracy condition, which is weaker than
Stein's non-degeneracy condition (see \cite[Remark 2.1.2]{Hytonen}).
Here is the main result of this article.

\begin{theorem}\label{thm:main}
Let $p\in(1,\infty)$, and let $T$ be a non-degenerate convolution
Calder\'on--Zygmund operator with its kernel satisfying
\eqref{eq:size}, \eqref{eq:holder}, and \eqref{eq:nd}.
Then, for every integer $m\ge3$, there exists a positive constant $C$
such that, for any $t\in[1,\infty)$,
\begin{equation}\label{eq:main-lower}
\sup_{[W]_{\mathcal A_p}\le t}
\left\| T\right\|_{L^p(W)\to L^p(W)}
\ge C t^{1+\frac{1}{p(p-1)}},
\end{equation}
where the supremum is taken over all matrix weights
$W:\ \mathbb R^n\to\mathbb C^{m\times m}$ with their weight
constants $[W]_{\mathcal A_p}\le t$.
For every integer $m\ge2$, the same conclusion holds
under the additional assumption that the odd part
\begin{equation*}
K_-(\cdot):=\frac{K(\cdot)-K(-\cdot)}2
\end{equation*}
itself satisfies the non-degeneracy condition \eqref{eq:nd}.
\end{theorem}

\begin{corollary} \label{cor:homog}
For the Hilbert transform and the Riesz transforms,
\eqref{eq:main-lower} holds for all $p\in(1,\infty)$ with $m\ge2$.
\end{corollary}

Corollary \ref{cor:homog} with the Hilbert transform
coincides with the main result \cite[Theorem 1.1]{DPTV,jqwx}.

To prove Theorem~\ref{thm:main}, we first decompose the kernel
into its odd and even parts $K=K_-+K_+$, where $K_+(\cdot):=\frac{K(\cdot)+ K(-\cdot)}{2}$.
Combining the non-degeneracy of $K$ with Brown's lemma,
we show that there exists $\tau\in\{-,+\}$ such that $K_\tau$
satisfies a lower bound along a chain of suitable scales; see \eqref{eq:sample}.
Based on these scales, we construct a matrix weight
$W:=(L^*L)^{\frac p2}$; see Proposition~\ref{prop:matrix-weight}.
Here $L:=\operatorname{diag}(w_0^{\frac1p},\ldots,w_q^{\frac1p})
e^{-\Phi\mathfrak J_q}$,
where $\{w_i\}_{i=0}^q$ is a family of scalar weights,
$\mathfrak J_q$ is a nilpotent shift matrix,
and $\Phi$ is a Cantor-type function.
We then choose a test function adapted to $W$.
Its norm $\|\vec f\|_{L^p(W)}$
reduces to an integral weighted by $w_0$; see \eqref{input}.
On the other hand, $\|T_\tau \vec f\|_{L^p(W)}$
is bounded below by an integral weighted by $w_q$;
see \eqref{eq:last-coordinate}.
The scalar weights are designed so that
$[W]_{\mathcal A_p}\sim [w_0]_{A_p}$,
as shown in Proposition~\ref{prop:matrix-weight},
while the contribution associated with $w_q$ is sufficiently large
to produce the desired sharp norm growth.
This yields the sharp lower bound for the operator $T_\tau$
associated with the kernel $K_\tau$; see Proposition~\ref{prop:norm-ratio}.
Finally, using the relation between $T$ and $T_\tau$,
together with a duality argument,
we transfer this lower bound to $T$ and thereby
complete the proof of Theorem~\ref{thm:main}.

The above proof differs from those of
Domelevo et al. \cite[Theorem 1.1]{DPTV} and Jiao et al. \cite[Theorem 1.1]{jqwx}
in the construction of the weight and in the passage to the Hilbert transform.
In the proof of \cite[Theorem 1.1]{DPTV}, successive rotations and stretches of eigenvalues
produce dyadic matrix weights and lower bounds for special dyadic operators.
Quasi-periodization and remodeling then transfer these
bounds to the Hilbert transform while controlling the characteristic
on arbitrary intervals. Jiao et al. \cite{jqwx} adapt this scheme to
general $p\in(1,\infty)$, coordinating the powers in the weight and its dual and
estimating the odd Haar shift difference $\mathbb S-\mathbb S^*$;
their transfer step also uses quasi-periodization and remodeling.

The proofs of Theorem \ref{thm:main} and Corollary \ref{cor:homog}
are given in Section \ref{sec2}.
To this end, in Subsection~\ref{sec:color},
we use the non-degeneracy of the kernel and Brown's lemma
to show that there exist $\tau\in\{-,+\}$
and a sequence of good scales along which
$K_\tau$ satisfies a uniform lower bound.
For any given such scale, in Subsection~\ref{sec:Cantor},
we construct a Cantor-type function $\Phi$
and a family $\{w_i\}_{i=0}^q$ of scalar weights.
In Subsection~\ref{sec:matrix weight},
we use $\Phi$, $\{w_i\}_{i=0}^q$, and nilpotent shift matrices
further to construct a matrix weight $W$
and compute its weight constant $[W]_{\mathcal A_p}$.
Finally, in Subsection~\ref{sec:test function},
via first choosing a suitable test function adapted to $W$
and establishing the sharp lower bound for
the operator $T_\tau$ whose kernel is $K_\tau$,
we then transfer this lower bound from $T_\tau$
to the original operator $T$,
thereby completing the proof of Theorem \ref{thm:main}.
As its application, we prove Corollary \ref{cor:homog}.

At the end of this introduction,
we make some notational conventions.
The notation $A\lesssim B$ means
that $A\leq CB$ for some positive constant $C$,
while $A\sim B$ means $A\lesssim B\lesssim A$.
If $f\leq Cg$ and $g=h$ or $g\leq h$,
we then write $f\lesssim g = h$ or $f\lesssim g\leq h$.
We use $\mathbf 0$ to denote the origin of $\mathbb R^n$.
Let $\mathbb Z$ denote the set of all integers,
$\mathbb N:=\{1,2,\ldots\}$, and
$\mathbb Z_+:=\mathbb N\cup\{0\}$.
If $E$ is a measurable set in $\mathbb R^n$,
then we denote by $\mathbf 1_E$ its characteristic function.
For any $t\in\mathbb R$, we use the \emph{notation $\lceil t \rceil$}
to denote the smallest integer not less than $t$.
Finally, in all subsequent proofs we
retain the notation introduced in the relevant statement.

\section{Proof of Theorem \ref{thm:main}} \label{sec2}

This section is devoted to the proof of Theorem~\ref{thm:main}.
The argument is organized into four parts.
In Subsection~\ref{sec:color}, from the kernel $K$ we extract
the parity component $K_\tau$ and the matched sequence of scales needed
for the subsequent construction.
In Subsection~\ref{sec:Cantor}, we use these scales to construct
a Cantor-type function $\Phi$ and a family of scalar weights
$\{w_i\}_{i=0}^q$.
In Subsection~\ref{sec:matrix weight}, these objects are assembled,
via nilpotent shift matrices, into a matrix weight $W$,
whose weight constant $[W]_{\mathcal A_p}$ is then estimated.
Subsection~\ref{sec:test function} completes
the proof of Theorem~\ref{thm:main} by first proving
the required lower bound for $T_\tau$ and
then transferring it to the original operator $T$.
We conclude the section with the proof of Corollary~\ref{cor:homog}.

\subsection{Kernel Non-Degeneracy and Scale Selection}
\label{sec:color}

In this subsection, we use the non-degeneracy
condition \eqref{eq:nd} and Brown's lemma to select arbitrarily long
sequences of dyadic scales with controlled gaps.
At these scales, a fixed phase, parity, and direction
yield a uniform non-degeneracy lower bound.
To this end, we first establish the following variant of Brown's lemma,
which when $A=1$ coincides with the original result \cite[Lemma 1]{Brown}.

\begin{lemma}\label{lem:color}
Suppose the non-negative integers are colored with finitely many colors.
Then there exist a color $\mathscr C$ and $B\in\mathbb N$ such that,
for any $A,J\in\mathbb N$, we can choose non-negative integers
$k_0<\cdots<k_J$ of color $\mathscr C$ satisfying,
for any $j\in\{0,\ldots,J-1\}$,
\begin{equation*}
A\le k_{j+1}-k_j\le A+B-1.
\end{equation*}
\end{lemma}

\begin{proof}
By Brown's lemma, we conclude that there exist
a color $\mathscr C$ and $B\in\mathbb N$ such that,
for any $A,J\in\mathbb N$,
we can choose non-negative integers
$k_0<\cdots<k_{AJ}$ of color $\mathscr C$ satisfying
$k_{j+1}-k_j\le B$ for all $j\in\{0,\ldots,AJ-1\}$.
Let $j_0:=0$.
For any $i\in\{1,\ldots,J\}$, let
$j_i:=\min\{j:\ k_j-k_{j_{i-1}}\geq A\}$.
Then the indices $k_{j_0}<\cdots<k_{j_J}$
are all of color $\mathscr C$ and satisfy
$A\le k_{j_{i+1}}-k_{j_i}\le A+B-1$ for all $i\in\{0,\ldots,J-1\}$.
This completes the proof of Lemma \ref{lem:color}.
\end{proof}

Let $\Lambda:=\{1,\mathrm i,-1,-\mathrm i\}$,
and let $\mathcal V$ be a finite
$\min\{(\frac{a}{2H_K})^{\frac{1}{\alpha}},\frac12\}$-net
of the compact annulus
\begin{equation*}
\left\{
x\in\mathbb R^n:\
1\le |x|\le \left(\frac{C_K}{a}\right)^{\frac{1}{n}}
\right\}.
\end{equation*}
In the remainder of this subsection,
we choose $\Lambda\times\{+,-\}\times\mathcal V$ as
the \emph{color set}, which is obviously finite.
Combining Lemma \ref{lem:color} with the properties of the kernel $K$,
we obtain the desired non-degeneracy lower bound
based on the color set $\Lambda\times\{+,-\}\times\mathcal V$.

\begin{proposition} \label{non-degeneracy lower bound}
There exist a color
$(\lambda,\tau,v)\in \Lambda\times\{+,-\}\times\mathcal V$
and $B\in\mathbb N$ such that,
for any $A,J\in\mathbb N$, we can choose integers
$0\le k_0<\cdots<k_J$ satisfying, for any $j\in\{0,\ldots,J-1\}$,
\begin{equation}\label{eq:color-gaps}
A\le k_{j+1}-k_j\le A+B-1
\end{equation}
and, for any $j\in\{0,\ldots,J\}$,
\begin{equation} \label{lower2}
\operatorname{Re}\{\lambda K_\tau(2^{-k_j}v)\}
\ge a_0|2^{-k_j}v|^{-n},
\end{equation}
where $a_0:=\frac{\sqrt2}{8}a$.
\end{proposition}

\begin{proof}
By Lemma \ref{lem:color}, to prove the present proposition
it suffices to show that, for any $k\in\mathbb Z_+$,
there exists $(\lambda_k,\tau_k,v_k)
\in \Lambda\times\{+,-\}\times\mathcal V$ such that
\begin{equation}\label{non-degeneracy}
\operatorname{Re}\{\lambda_k K_{\tau_k}(2^{-k}v_k)\}
\ge a_0|2^{-k}v_k|^{-n}.
\end{equation}

By the non-degeneracy condition \eqref{eq:nd},
we conclude that, for any $k\in\mathbb Z_+$,
there exists $x_k\in\mathbb R^n$ such that
\begin{equation}\label{eq:zk-nd}
|x_k|\ge 2^{-k}
\quad\text{and}\quad
|K(x_k)|\ge a2^{kn}.
\end{equation}
On the other hand, the size condition \eqref{eq:size} implies
$|K(x_k)|\le \frac{C_K}{|x_k|^n}$.
Combining this with \eqref{eq:zk-nd} yields
$a2^{kn}\le \frac{C_K}{|x_k|^n}$, and hence
$|x_k| \le (\frac{C_K}{a})^{\frac{1}{n}}2^{-k}$.
Using this and \eqref{eq:zk-nd}, we obtain
\begin{equation}\label{eq:zk-annulus}
1\le |2^kx_k|\le \left(\frac{C_K}{a}\right)^{\frac{1}{n}}.
\end{equation}
Then there exists $v_k\in\mathcal V$ such that
\begin{equation} \label{delta}
|2^kx_k-v_k|
< \min\left\{\left( \frac{a}{2 H_K} \right)^{\frac{1}{\alpha}},
\frac12\right\}.
\end{equation}
From this and \eqref{eq:zk-nd}, we deduce that
$$
|x_k-2^{-k}v_k|
=2^{-k}|2^kx_k-v_k|
< 2^{-k} \frac12
\le\frac{|x_k|}{2},
$$
so the H\"older regularity condition \eqref{eq:holder} of $K$ implies
\begin{equation} \label{est1}
|K(2^{-k}v_k)-K(x_k)|
\le H_K \frac{|x_k-2^{-k}v_k|^\alpha}{|x_k|^{n+\alpha}}
=  H_K \frac{|2^kx_k-v_k|^\alpha}{|2^kx_k|^{n+\alpha}} 2^{kn}
< \frac{a}{2} 2^{kn},
\end{equation}
where the last inequality follows from \eqref{eq:zk-annulus} and \eqref{delta}.
By \eqref{est1} and \eqref{eq:zk-nd}, we conclude that
\begin{align*}
|K(2^{-k}v_k)|
\geq |K(x_k)| - |K(2^{-k}v_k)-K(x_k)|
> \frac{a}{2} 2^{kn}.
\end{align*}
Using this and $K=K_++K_-$, we obtain, for any $k\in\mathbb Z_+$,
\begin{equation*}
\frac{a}{2} 2^{kn}
< |K(2^{-k}v_k)|
\le |K_+(2^{-k}v_k)|+|K_-(2^{-k}v_k)|.
\end{equation*}
Thus, there exists $\tau_k\in\{+,-\}$ such that
\begin{equation}\label{eq:parity-nd-scale}
|K_{\tau_k}(2^{-k}v_k)|
> \frac a4 2^{kn}.
\end{equation}

Note that, for any $z\in\mathbb C$,
\begin{align*}
\max_{\lambda\in\Lambda} \operatorname{Re} (\lambda z)
= \max\{|\operatorname{Re} z|,|\operatorname{Im} z|\}
\geq \frac{\sqrt2}2 |z|.
\end{align*}
Thus, there exists $\lambda_k\in\Lambda$ such that
\begin{equation*}
\operatorname{Re}\{\lambda_k K_{\tau_k}(2^{-k}v_k)\}
\ge \frac{\sqrt2}{2} |K_{\tau_k}(2^{-k}v_k)|
> \frac{\sqrt2 a}{8} 2^{kn},
\end{equation*}
where the last inequality used \eqref{eq:parity-nd-scale}.
This, together with the fact that $|v_k|\ge1$ ($v_k\in\mathcal V$),
further implies \eqref{non-degeneracy},
which completes the proof of Proposition \ref{non-degeneracy lower bound}.
\end{proof}

\begin{remark} \label{key}
In Proposition \ref{non-degeneracy lower bound},
if the odd part $K_-$ itself satisfies the
non-degeneracy condition \eqref{eq:nd},
then the preceding argument can be applied directly to $K_-$.
Consequently, the parity $\tau$
in Proposition \ref{non-degeneracy lower bound} can be fixed as $-$.
\end{remark}

Finally, for convenience in later applications, we reformulate 
the non-degeneracy lower bound \eqref{lower2} of
Proposition \ref{non-degeneracy lower bound}
by expressing $2^{-k_j}v$, the variable of $K_\tau$, in polar coordinates.
To this end, let
\begin{equation*}
A:=\max\left\{4,\left\lceil
2+\frac1\alpha\log_2\left(\frac{2H_K}{a_0}\right)
\right\rceil\right\}.
\end{equation*}
Then $A\ge4$ and
\begin{equation}\label{eq:A-choice}
H_K(4\cdot2^{-A})^\alpha
\le \frac{a_0}{2}.
\end{equation}
Fix the unit vector
\begin{equation} \label{good direction}
\vartheta:=\frac{v}{|v|}.
\end{equation}
Let $L_0:=|v|2^{-k_0}$ and, for any $j\in\mathbb N$,
$L_j:= \kappa_j L_{j-1}$, where
\begin{equation}\label{eq:scales}
\kappa_j:=
\begin{cases}
2^{-k_j+k_{j-1}} &\text{if } 1\le j\le J,\\
2^{-A} &\text{if } j\ge J+1.
\end{cases}
\end{equation}
Then \eqref{eq:color-gaps} implies that, for any $j\in\mathbb N$,
\begin{equation*}
\kappa_*:=2^{-(A+B-1)}
\le \kappa_j
\le 2^{-A}=:\kappa^*
\le\frac1{16}.
\end{equation*}
Note that, for any $j\in\{0,\ldots,J\}$,
we have $L_j\vartheta=2^{-k_j}v$.
From this and \eqref{lower2}, it follows that
\begin{equation}\label{eq:sample}
\operatorname{Re}\{\lambda K_\tau(L_j\vartheta)\}
=\operatorname{Re}\{\lambda K_\tau(2^{-k_j}v)\}
\ge a_0|2^{-k_j}v|^{-n}
=a_0L_j^{-n},
\end{equation}
which is the desired reformulation of
the non-degeneracy lower bound \eqref{lower2} of
Proposition \ref{non-degeneracy lower bound}.
Let $q:=1$ if $\tau=-$ and $q:=2$ if $\tau=+$.
The construction below uses $(q+1)\times(q+1)$ matrices.

\subsection{Cantor-Type Functions and Scalar Weights}
\label{sec:Cantor}

In this subsection, we construct a variable-ratio Cantor set
using the parameters introduced in Subsection \ref{sec:color},
and then use this Cantor set to define
a Cantor-type function and scalar weights
that constitute the desired matrix weight.

Recall that $q\in\{1,2\}$, and we assume $p\in[2,\infty)$. Let
\begin{equation}\label{eq:params}
\varepsilon\in\left(0, \frac1{10}\right),\quad
\delta:=\varepsilon^{p-1},\quad\text{and}\quad
\eta:=\frac{(p-1)\varepsilon+\delta}{p}.
\end{equation}
Then $\delta\le\eta\le\varepsilon$.

Let $\{L_j\}_{j\in\mathbb Z_+}$ be as in \eqref{eq:scales}.
We construct the variable-ratio Cantor set as follows.
Starting with $[0,L_0]$, at each generation $j\in\mathbb Z_+$,
we retain two subintervals of length $L_{j+1}$
within every interval of generation $j$,
one sharing its left endpoint and the other sharing its right endpoint.
The contraction ratios $\kappa_j$ are as in \eqref{eq:scales}.
Let $\mathcal I_j$ denote the family of retained intervals of
generation $j$ and $\mathfrak C$ the resulting Cantor set. Let
\begin{equation} \label{Cantor set}
E:=\{t\vartheta:\ t\in\mathfrak C\}\subset\mathbb R^n
\quad\text{and}\quad
\rho(x):=\operatorname{dist}(x,E)
\quad\text{for all } x\in\mathbb R^n,
\end{equation}
where $\vartheta$ is as in \eqref{good direction}.

\begin{figure}[H]
\centering
\begin{tikzpicture}[x=0.95cm,y=1cm,font=\small]
\node[left,text=black] at (-0.25,0) {$E_{0}$};

\node[left,text=myred] at (-0.25,-0.7) {$E_{1}$};

\node[left,text=myblue] at (-0.25,-1.4) {$E_{2}$};

\node[left,text=mygreen] at (-0.25,-2.1) {$E_{3}$};

\draw[gray!70,->,>=stealth',thin] (12.15,0) -- (12.8,0) node[right] {$\vartheta$};
\draw[black,line width=1.5pt] (0,0) -- (12,0);

\foreach \a in {0,9}
\draw[myred,line width=1.5pt] (\a,-0.7) -- ({\a+3},-0.7);

\foreach \a in {0,2.25,9,11.25}{
\draw[myblue,line width=1.5pt] (\a,-1.4) -- ({\a+0.75},-1.4);

\draw[mygreen,line width=1.5pt] (\a,-2.1) -- ({\a+0.1875},-2.1);
\draw[mygreen,line width=1.5pt] ({\a+0.5625},-2.1) -- ({\a+0.75},-2.1);
}

\node[above] at (0,0.08) {$\mathbf 0$};
\node[above] at (12,0.08) {$L_0\vartheta$};

\node at (6,-2.6) {%
  \tikz[baseline=-0.6ex]{
    \foreach \y in {0,0.16,0.32}
      \fill (0,\y) circle[radius=1.2pt];
  }%
};

\node at (6,-3.6)
{$E=\displaystyle\bigcap_{j=0}^\infty E_j,
\qquad E_j=\bigcup_{I\in\mathcal I_j}\{t\vartheta:t\in I\}$};
\end{tikzpicture}
\caption{The Cantor set $E$ on the axis $\mathbb R\vartheta:=\{t\vartheta:\ t\in\mathbb R\}$.
The segments are finite-stage retained intervals, not the limiting
set itself. Their relative lengths are enlarged for visibility,
and the diagram is not to scale.}
\end{figure}

Let $j\in\mathbb Z_+$. A simple geometric observation shows that
\begin{equation*}
V_j:=2^jL_j^n\sim |\{x\in\mathbb R^n:\ \rho(x)\le L_j\}|.
\end{equation*}
Define the two endpoint weights
\begin{equation*}
w_{0,j}:=V_j^{p-1}e^{(p-1)\varepsilon j},
\qquad w_{q,j}:=\delta V_j^{-1}e^{-\delta j},
\end{equation*}
and, for any $k\in\mathbb Z\cap[0,q]$,
\begin{equation}\label{eq:wk-expand}
w_{k,j}
:=w_{0,j}^{1-\frac{k}{q}}w_{q,j}^{\frac{k}{q}}
=\delta^{\frac{k}{q}}V_j^{p(1-\frac{k}{q})-1}
e^{[(p-1)\varepsilon(1-\frac{k}{q})-\frac{k}{q}\delta]j}.
\end{equation}
Let
\begin{equation*}
A_j:=\left(\delta\frac{w_{0,j}}{w_{q,j}}\right)^{\frac1{pq}}
=\bigl(V_je^{\eta j}\bigr)^{\frac1q}.
\end{equation*}
We record for repeated use the identities
\begin{equation}\label{eq:key-identities}
\begin{cases}
V_jw_{0,j}^{-\frac{1}{p-1}}=e^{-\varepsilon j},\quad
V_jw_{q,j}=\delta e^{-\delta j},\\
\displaystyle\frac{w_{k,j}}{w_{l,j}}
=\delta^{\frac{k-l}{q}}A_j^{-(k-l)p},\qquad 0\le l\le k\le q.
\end{cases}
\end{equation}
Since $q\le2$ and $2\kappa_{j+1}^n\le\frac18$, we deduce that
\begin{equation}\label{eq:A-decay}
\frac{2A_{j+1}}{A_j}
=2\bigl(2\kappa_{j+1}^ne^\eta\bigr)^{\frac1q}
\le2\left(\frac{e^{\frac1{10}}}8\right)^{\frac12}<1.
\end{equation}

Let $\psi$ be the continuous non-decreasing piecewise linear function defined by
\begin{equation*}
\psi(s):=
\begin{cases}
0&\displaystyle\text{if } s\le\frac18,\\
\displaystyle2s-\frac14&\displaystyle\text{if } \frac18<s<\frac38,\\
\displaystyle\frac12&\displaystyle\text{if } \frac38\le s\le\frac58,\\
\displaystyle2s-\frac34&\displaystyle\text{if } \frac58<s<\frac78,\\
1&\displaystyle\text{if } s\ge\frac78.
\end{cases}
\end{equation*}
For $I:=[a_I,a_I+L_j]\in\mathcal I_j$, let $\alpha_I:=a_I+L_{j+1}$,
\begin{equation*}
\beta_I:=a_I+L_j-L_{j+1},\quad
G_I:=(\alpha_I,\beta_I),\quad\text{and}\quad
\Delta_j:=A_j-2A_{j+1}.
\end{equation*}
For any $t\in\mathbb R$, define
\begin{equation*}
\Phi(t):=\sum_{j=0}^{\infty}\Delta_j
\sum_{I\in\mathcal I_j}
\psi\left(\frac{t-\alpha_I}{\beta_I-\alpha_I}\right).
\end{equation*}
It follows from \eqref{eq:A-decay} that the series
$\sum_{j=0}^\infty 2^jA_j$ converges, and hence
\begin{equation*}
\sum_{j=0}^\infty 2^j\Delta_j
=\sum_{j=0}^\infty (2^jA_j-2^{j+1}A_{j+1})
=A_0.
\end{equation*}
Thus, the series in the definition of $\Phi$
converges absolutely and uniformly,
$\Phi$ is continuous and non-decreasing, and
\begin{equation*}
0\le\Phi\le A_0,\qquad
\Phi(t)=0\ (t\le0),\qquad
\Phi(t)=A_0\ (t\ge L_0).
\end{equation*}
For any $j\in\mathbb Z_+$ and $I\in\mathcal I_j$,
\begin{align}\label{eq:Phi-interval-osc}
\operatorname{osc}_I\Phi
&:=\sup_{s,t\in I}|\Phi(s)-\Phi(t)|
=\Phi(a_I+L_j)-\Phi(a_I) \notag \\
&\phantom{:}=\sum_{h=j}^{\infty} \Delta_h 2^{h-j}
= 2^{-j} \sum_{h=j}^\infty (2^hA_h-2^{h+1}A_{h+1})
= A_j.
\end{align}
We keep the \emph{symbol $\Phi$} for the extension,
i.e., for any $x\in\mathbb R^n$,
\begin{equation} \label{Phi}
\Phi(x):=\Phi((x,\vartheta)),
\end{equation}
where $(x,\vartheta)$ denotes the standard inner product on $\mathbb R^n$.
For any $t\in\mathbb R$, let $r(t):=\operatorname{dist}(t,\mathfrak C)$.

\begin{lemma}\label{lem:phi-modulus}
Let $i\in\mathbb Z_+$ and $ R\in (L_{i+1}, L_i]$.
Then, for any $s,t\in\mathbb R$ with
$r(s),r(t)\le\frac32R$ and $|s-t|\le R$,
we have $|\Phi(s)-\Phi(t)|\le3A_i$.
\end{lemma}

\begin{proof}
Choose nearest points $s_0,t_0\in\mathfrak C$ to $s,t$, respectively.
Then
$$
|s_0-t_0|
\leq |s_0-s|+|s-t|+|t-t_0|
\le 4R
\le 4L_i.
$$
Since $\kappa^*\leq \frac1{16}$, we infer that the
distinct intervals of generation $i$ are at least $14L_i$
apart, so $s_0$ and $t_0$ lie in the same interval of generation $i$.
By this and \eqref{eq:Phi-interval-osc},
we obtain $|\Phi(s_0)-\Phi(t_0)|\le A_i$.

If $s\in\mathfrak C$, or if $s\notin[0,L_0]$, or if $s$ lies in a constant
end piece of some gap, then $\Phi(s)=\Phi(s_0)$. Otherwise $s$ lies in
some gap $G_I$ of some generation $j$ with
$r(s)\ge\frac18|G_I|\ge\frac7{64}L_j$.
If $j<i$, then $L_j\ge16L_i$, and hence
$r(s)\ge\frac{7}4L_i>\frac{3}2R$,
which contradicts the assumption $r(s)\le \frac32 R$.
Thus, $j\ge i$ and
$$
|\Phi(s)-\Phi(s_0)|\le\Delta_j\le A_j\le A_i.
$$
The same bound holds for $t$. Combining the above estimates, we obtain
$$
|\Phi(s)-\Phi(t)|
\leq |\Phi(s)-\Phi(s_0)|
+|\Phi(s_0)-\Phi(t_0)|
+|\Phi(t_0)-\Phi(t)|
\leq 3A_i.
$$
This completes the proof of Lemma \ref{lem:phi-modulus}.
\end{proof}

For $I\in\mathcal I_j$, let $P_I$ be the middle quarter of its gap, namely
\begin{equation*}
P_I:=\left(\alpha_I+\frac38(\beta_I-\alpha_I),
\alpha_I+\frac58(\beta_I-\alpha_I)\right),
\end{equation*}
and define the tube
\begin{equation*}
H_I:=\left\{t\vartheta+z:\ t\in P_I,\
z\in\mathbb R^n,\ (z,\vartheta)=0,\ |z|<\frac{L_j}{10}\right\}
\quad\text{and}\quad
H_j:=\bigcup_{I\in\mathcal I_j}H_I.
\end{equation*}

\begin{figure}[H]
\centering
\begin{tikzpicture}[x=0.95cm,y=0.95cm,font=\small]
\draw[myblue,dashed,line width=1.5pt] (4.65,-1.2) rectangle (7.35,1.2);
\node[text=myblue] at (5.85,0.7) {$H_I$};

\draw[gray!70,->,>=stealth',thin] (-0.8,0) -- (13.2,0)
node[right] {$\mathbb R\vartheta$};
\draw[black,line width=1.5pt] (0,0) -- (0.6,0);
\node[below left] at (0,-0.05) {$a_I$};
\node[below right] at (0.6,-0.05) {$\alpha_I$};

\draw[black,line width=1.5pt] (11.4,0) -- (12,0);
\node[below right] at (12,-0.05) {$a_I+L_j$};
\node[below left] at (11.4,-0.05) {$\beta_I$};

\draw[myred,line width=1.5pt] (4.65,0) -- (7.35,0);
\foreach \a in {4.65,7.35}
\draw[myred,fill=white] (\a,0) circle (1.5pt);
\node[below,text=myred] at (5.85,-0.05) {$P_I$};

  \draw[<->] (8.1,0) -- (8.1,1.2)
    node[midway,right] {$\frac{L_j}{10}$};
  \draw[gray,dotted] (7.35,1.2) -- (8.25,1.2);
  \draw[decorate,line width=1.5pt,decoration={brace,amplitude=4pt}] (0.6,1.7) -- (11.4,1.7)
    node[midway,above=6pt] {$G_I=(\alpha_I,\beta_I)$};
  \draw[decorate,line width=1.5pt,decoration={brace,amplitude=4pt}] (7.35,-1.5) -- (4.65,-1.5)
    node[midway,below=6pt] {$|P_I|=\frac14(\beta_I-\alpha_I)$};
\end{tikzpicture}
\caption{An axial section of $H_I$.
For $n=1$, $H_I$ reduces to $\{t\vartheta:\ t\in P_I\}$.}
\end{figure}

Distinct gaps have disjoint axial projections, so the tubes $H_I$ are pairwise disjoint.
Observe that the function $\Phi$ is constant on each $H_I$,
\begin{equation}\label{eq:plateau-volume}
|H_I|\sim L_j^n,
\quad\text{and}\quad
|H_j|\sim2^jL_j^n=V_j.
\end{equation}

\begin{lemma}\label{lem:plateau-distance}
For any $j\in\mathbb Z_+$ and $x\in H_j$, we have $L_{j+1}<\rho(x)<L_j$,
where $\rho$ is as in \eqref{Cantor set}.
\end{lemma}

\begin{proof}
If $x=t\vartheta+z\in H_I$ with $I\in\mathcal I_j$, then
\begin{equation*}
\frac{21}{64}L_j
\le\frac38(L_j-2L_{j+1})\le r(t)
\le\frac12(L_j-2L_{j+1})\le\frac12L_j.
\end{equation*}
This, together with $\rho(x)^2=r(t)^2+|z|^2$, further implies that
\begin{equation*}
\rho(x)\ge\frac{21}{64}L_j>L_{j+1}
\quad\text{and}\quad
\rho(x)<\sqrt{\frac14+\frac1{100}} L_j<L_j,
\end{equation*}
which completes the proof of Lemma \ref{lem:plateau-distance}.
\end{proof}

Extend the layer-wise weights to the whole space by setting,
for any $k\in\mathbb Z\cap[0,q]$ and $x\in\mathbb R^n$,
\begin{equation}\label{eq:outside}
w_k(x):=
\begin{cases}
w_{k,0} &\text{if } \rho(x)>L_0\ \text{or}\ \rho(x)=0,\\
w_{k,j} &\text{if } L_{j+1}<\rho(x)\le L_j \text{ for some } j\in\mathbb Z_+,
\end{cases}
\end{equation}
where $w_{k,j}$ is as in \eqref{eq:wk-expand}.
To estimate the integral of $w_k$, we need the following lemma.
Here, and thereafter, for any $i,j\in\mathbb Z$,
let $\mathbf 1_{\{i=j\}}:=1$ if $i=j$
and let $\mathbf 1_{\{i=j\}}:=0$ otherwise.

\begin{lemma}\label{lem:tail}
Let $k\in\mathbb Z\cap[0,q]$ and $i\in\mathbb Z_+$. Then
\begin{align*}
\sum_{j=i}^{\infty}2^{j-i}L_j^nw_{k,j}
\sim L_i^nw_{k,i}\delta^{-\mathbf1_{\{k=q\}}}
\end{align*}
and
\begin{align*}
\sum_{j=i}^{\infty}2^{j-i}L_j^nw_{k,j}^{-\frac{1}{p-1}}
\sim L_i^nw_{k,i}^{-\frac{1}{p-1}}
\varepsilon^{-\mathbf1_{\{k=0\}}},
\end{align*}
where the positive equivalence constants are absolute.
\end{lemma}

\begin{proof}
Let $\theta:=\frac{k}{q}$. By \eqref{eq:wk-expand}, we have
\begin{align*}
2^j L_j^n w_{k,j}
=V_jw_{k,j}
=\delta^\theta V_j^{p(1-\theta)}
e^{[(p-1)\varepsilon(1-\theta)-\theta\delta]j}
\end{align*}
and
\begin{align*}
2^j L_j^n w_{k,j}^{-\frac{1}{p-1}}
=V_jw_{k,j}^{-\frac{1}{p-1}}
=\delta^{-\frac{\theta}{p-1}}V_j^{p'\theta}
e^{[-\varepsilon(1-\theta)+\frac{\theta\delta}{p-1}]j}.
\end{align*}
For $k=q$ the first expression equals $\delta e^{-\delta j}$, and for
$k=0$ the second equals $e^{-\varepsilon j}$.
Combining this and $1-e^{-s}\sim s$ for $0<s\le1$,
we obtain the desired estimates.

For $k<q$, the ratio of consecutive terms of the first sum is at most
\begin{equation*}
\left[8^{-p}e^{(p-1)\varepsilon}\right]^{1-\theta}
\le\left(\frac{e^{\frac1{10}}}8\right)^{p(1-\theta)}
\le\frac{e^{\frac1{10}}}8<1,
\end{equation*}
where we used $p(1-\theta)\ge1$.
For $k>0$, the ratio of consecutive terms of the second sum is at most
\begin{equation*}
8^{-p'\theta}e^{\frac{\theta\delta}{p-1}}
=\left(\frac{e^{\frac{\delta}{p}}}8\right)^{p'\theta}
\le\left(\frac{e^{\frac1{10}}}8\right)^{\frac12}<1
\end{equation*}
because $p'\theta\ge\frac12$. Hence in both cases the tail sum is
comparable to its first term.
This completes the proof of Lemma \ref{lem:tail}.
\end{proof}

\begin{lemma}\label{lem:all-cubes}
Let $l,k\in\mathbb Z$ satisfy $0\le l\le k\le q$.
Then, for any cube $Q\subset\mathbb R^n$,
\begin{equation}\label{eq:moment-oscillation}
\bigl(\operatorname{osc}_Q\Phi\bigr)^{(k-l)p}
\left(\fint_Qw_k\right)
\left(\fint_Qw_l^{-\frac{1}{p-1}}\right)^{p-1}
\lesssim \delta^{\frac{k-l}{q}-\mathbf1_{\{k=q\}}-\mathbf1_{\{l=0\}}},
\end{equation}
where the implicit positive constant is independent of $Q$,
$\varepsilon$, $J$, and the choice of $\{k_i\}_{i=0}^J$.
\end{lemma}

\begin{proof}
For any cube $Q\subset\mathbb R^n$, we claim that
there exist a generation $i\in\mathbb Z_+$
and a positive constant $C$ such that
\begin{align}\label{eq:osc-ref}
\operatorname{osc}_Q\Phi\le C A_i,
\end{align}
\begin{align}\label{eq:mom-pos}
\fint_Qw_k\le Cw_{k,i}\delta^{-\mathbf1_{\{k=q\}}},
\end{align}
and
\begin{align}\label{eq:mom-neg}
\fint_Qw_k^{-\frac{1}{p-1}}
\le Cw_{k,i}^{-\frac{1}{p-1}}\varepsilon^{-\mathbf1_{\{k=0\}}}.
\end{align}
Multiplying the three estimates and using \eqref{eq:key-identities}
together with $\varepsilon^{p-1}=\delta$ gives
\eqref{eq:moment-oscillation}.
Next, we prove this claim by considering three cases for $Q$.

\emph{Case (1)} $\ell(Q)>(3\sqrt n)^{-1}L_0$.
In this case, take $i=0$.
Note that $w_k=w_{k,0}$ on $\{x\in\mathbb R^n:\ \rho(x)>L_0\}$.
For any $j\in\mathbb Z_+$, let
$$
\Omega_j:=\{x\in\mathbb R^n:\ L_{j+1}<\rho(x)\le L_j\}.
$$
Then
$$
|\Omega_j|\le |\{x\in\mathbb R^n:\ \rho(x)\le L_j\}| \sim V_j.
$$
By this and Lemma \ref{lem:tail}, we conclude that
\begin{align*}
\int_{\{0<\rho\le L_0\}}w_k
\lesssim \sum_{j=0}^\infty V_jw_{k,j}
\sim L_0^nw_{k,0}\delta^{-\mathbf1_{\{k=q\}}}
\end{align*}
and
\begin{align*}
\int_{\{0<\rho\le L_0\}}w_k^{-\frac{1}{p-1}}
\lesssim \sum_{j=0}^\infty V_jw_{k,j}^{-\frac{1}{p-1}}
\sim L_0^nw_{k,0}^{-\frac{1}{p-1}}
\varepsilon^{-\mathbf1_{\{k=0\}}}.
\end{align*}
These, together with $|Q|\ge(3\sqrt n)^{-n}L_0^n$, further imply that
\begin{align*}
\fint_Qw_k
&= \frac1{|Q|}\int_{Q\cap \{0<\rho\le L_0\}}w_k
+ \frac1{|Q|}\int_{Q\cap \{\rho > L_0\}}w_k \\
&\lesssim w_{k,0} \delta^{-\mathbf1_{\{k=q\}}} + w_{k,0}
\sim w_{k,0} \delta^{-\mathbf1_{\{k=q\}}},
\end{align*}
which is precisely \eqref{eq:mom-pos}.
The same argument yields \eqref{eq:mom-neg}.
Moreover, since $0\le\Phi\le A_0$,
it follows that \eqref{eq:osc-ref} holds as well.

\emph{Case (2)} $\ell(Q)\le(3\sqrt n)^{-1}L_0$ and
$\operatorname{dist}(Q,E)<2\sqrt n\ell(Q)$.
In this case, take the unique $i\in\mathbb Z_+$ such that
\begin{equation*}
L_{i+1}<R:=3\sqrt n\ell(Q)\le L_i.
\end{equation*}
Then, for any $x\in Q$,
$$
\rho(x)
\leq \operatorname{dist}(Q,E) +\sqrt{n} \ell(Q)
<R\le L_i.
$$
Thus, $Q\cap\Omega_j\ne\emptyset$ implies $j\ge i$.
For these $j$, we have
\begin{equation}\label{eq:local-layer-volume}
|Q\cap\Omega_j|\lesssim 2^{j-i}L_j^n.
\end{equation}
Indeed, if $x\in Q$ and $\rho(x)\le L_j$, then the axial coordinate of
its nearest point in $E$ lies in the interval obtained by enlarging the
axial projection of $Q$ by $L_j$ on either side. That interval has
length at most
$\operatorname{diam}(Q)+2L_j\le\frac{7L_i}3$, and therefore meets at most
one interval of generation $i$, since distinct intervals of generation
$i$ are at least $14L_i$ apart (and there is only one when $i=0$).
This interval has only $2^{j-i}$ descendants of generation $j$, and the
$L_j$-neighborhood of each descendant is contained in a ball of radius
$2L_j$. This proves \eqref{eq:local-layer-volume}.

Using \eqref{eq:local-layer-volume} and Lemma \ref{lem:tail}, we obtain
\begin{align*}
\int_Qw_k
\lesssim \sum_{j=i}^\infty 2^{j-i}L_j^n w_{k,j}
\sim L_i^nw_{k,i}\delta^{-\mathbf1_{\{k=q\}}}
\end{align*}
and
\begin{align*}
\int_Qw_k^{-\frac{1}{p-1}}
\lesssim \sum_{j=i}^\infty 2^{j-i}L_j^n w_{k,j}^{-\frac{1}{p-1}}
\sim L_i^nw_{k,i}^{-\frac{1}{p-1}} \varepsilon^{-\mathbf1_{\{k=0\}}}.
\end{align*}
This, together with $\ell(Q)\sim L_i$, further implies
\eqref{eq:mom-pos} and \eqref{eq:mom-neg}.
For $x,y\in Q$, the projections $s:=(x,\vartheta)$ and
$t:=(y,\vartheta)$ satisfy
$r(s),r(t)\le R$ and $|s-t|\le R$.
Thus, the oscillation estimate \eqref{eq:osc-ref} follows from Lemma \ref{lem:phi-modulus}.

\emph{Case (3)} $\ell(Q)\le(3\sqrt n)^{-1}L_0$ and
$R:=\operatorname{dist}(Q,E)\ge2\sqrt n\ell(Q)$.
In this case, for any $x\in Q$,
\begin{equation*}
R
\le \rho(x)
\le \operatorname{dist}(Q,E) + \operatorname{diam}(Q)
\le \frac32R.
\end{equation*}
If $R>L_0$, take $i=0$ and, in this case,
the weights are constant on $Q$ and the oscillation is at most $A_0$.
If $R\le L_0$, take $L_{i+1}<R\le L_i$.
For $i\ge1$ we have $\frac{3L_i}2<L_{i-1}$, so $Q$ can meet only layers $i$
and $i-1$; for $i=0$ only layer $0$ and the constant region are
involved.
By \eqref{eq:wk-expand}, consecutive weight values satisfy
\begin{equation*}
\frac{w_{k,j+1}}{w_{k,j}}
=(2\kappa_{j+1}^n)^{p(1-\frac{k}{q})-1}
e^{(p-1)\varepsilon(1-\frac{k}{q})-\frac{k}{q}\delta}.
\end{equation*}
Both this ratio and its reciprocal are uniformly bounded,
with bounds depending only on $p$, $n$, and $\kappa_*$.
Therefore, $w_k\sim w_{k,i}$ on $Q$, which yields \eqref{eq:mom-pos} and \eqref{eq:mom-neg}.
For any $x,y\in Q$, let $s:=(x,\vartheta)$ and $t:=(y,\vartheta)$. Then
$$
r(s)
=\inf_{u\in\mathfrak C}|s-u|
=\inf_{u\in\mathfrak C}|(x-u\vartheta,\vartheta)|
\le \inf_{u\in\mathfrak C}|x-u\vartheta|
=\rho(x)
\le\frac32R.
$$
The same argument yields $r(t)\le\frac32R$. Moreover,
$$
|s-t|
=|(x-y,\vartheta)|
\le |x-y|
\le \operatorname{diam}(Q)\le\frac12R<R.
$$
These, together with Lemma \ref{lem:phi-modulus} again,
yield the oscillation estimate \eqref{eq:osc-ref},
which completes the proof of the claim and hence Lemma \ref{lem:all-cubes}.
\end{proof}

\subsection{Special Matrix Weights from Nilpotent Shift Matrices}
\label{sec:matrix weight}

Let $e_0,\ldots,e_q$ be the standard basis of $\mathbb C^{q+1}$,
where $e_k$ is the vector whose $(k+1)$-th coordinate is $1$ and all other coordinates are $0$.
Let $\mathfrak J_q$ be the nilpotent shift matrix defined by
$\mathfrak J_qe_k=e_{k+1}$ for any $k\in\{0,1,\ldots,q-1\}$ and $\mathfrak J_qe_q=0$.
Then $\mathfrak J_q^{q+1}$ is the zero matrix and, for any $t\in \mathbb R$,
\begin{equation*}
e^{t\mathfrak J_q}
=\sum_{i=0}^q \frac{t^i}{i!}\mathfrak J_q^i.
\end{equation*}

\begin{proposition}\label{prop:matrix-weight}
Let $p\in[2,\infty)$, $\{w_k\}_{k=0}^q$ be as in \eqref{eq:outside},
and $\Phi$ as in \eqref{Phi}. Let
\begin{equation*}
D:=\operatorname{diag}(w_0^{\frac1p},\ldots,w_q^{\frac1p}),\quad
L:=De^{-\Phi\mathfrak J_q},\quad\text{and}\quad
W:=(L^*L)^{\frac p2}.
\end{equation*}
Then $W\in \mathcal A_p$ and
\begin{equation}\label{eq:characteristics}
[W]_{\mathcal A_p}
\sim [w_0]_{A_p}
\sim\delta^{-1}
=\varepsilon^{1-p},
\end{equation}
where $\varepsilon$ is as in \eqref{eq:params}
and the positive equivalence constants are independent of $\varepsilon$.
\end{proposition}

\begin{proof}
Since $W^{\frac2p}=L^*L$, we deduce that
$|W^{\frac1p}\vec z|=|L\vec z|$ for all $\vec z\in\mathbb C^{q+1}$,
and hence, for any $x,y\in\mathbb R^n$,
\begin{equation*}
\left\|W^{\frac1p}(x)W^{-\frac1p}(y)\right\|
=\left\|L(x)L^{-1}(y)\right\|.
\end{equation*}
Note that, for any $x,y\in\mathbb R^n$,
$L(x)L^{-1}(y)=D(x)e^{[\Phi(y)-\Phi(x)]\mathfrak J_q}D(y)^{-1}$,
whose nonzero entries are
\begin{equation*}
[L(x)L^{-1}(y)]_{k,l}
=[w_k(x)]^{\frac1p}[w_l(y)]^{-\frac1p}
\frac{[\Phi(y)-\Phi(x)]^{k-l}}{(k-l)!},
\qquad 0\le l\le k\le q.
\end{equation*}
From this, the equivalence of norms on finite-dimensional spaces,
and Lemma \ref{lem:all-cubes}, we infer that,
for any cube $Q\subset\mathbb R^n$,
\begin{align*}
&\fint_Q\left[\fint_Q
\left\|W^{\frac1p}(x)W^{-\frac1p}(y)\right\|^{p'}\,dy
\right]^{\frac{p}{p'}}dx\\
&\quad\lesssim \sum_{0\le l\le k\le q}
(\operatorname{osc}_Q\Phi)^{(k-l)p}
\left(\fint_Qw_k\right)
\left(\fint_Qw_l^{-\frac{1}{p-1}}\right)^{p-1}\\
&\quad\lesssim \sum_{0\le l\le k\le q}
\delta^{\frac{k-l}{q}-\mathbf1_{\{k=q\}}-\mathbf1_{\{l=0\}}}
\sim \delta^{-1}.
\end{align*}
Taking the supremum over all cubes $Q\subset\mathbb R^n$,
we obtain $[W]_{\mathcal A_p}\lesssim \delta^{-1}$.

It remains to prove the reverse estimate. Take
$Q_0=[-3L_0,3L_0]^n$.
Then all tubes $H_I$ lie in $B(\mathbf 0,\frac{11L_0}{10})\subset Q_0$,
$$
Q_0\setminus B(\mathbf 0,2L_0)\subset\{x\in\mathbb R^n:\ \rho(x)\ge L_0\},
\quad\text{and}\quad
|Q_0\setminus B(\mathbf 0,2L_0)|\sim |Q_0|.
$$
Thus,
\begin{equation*}
\fint_{Q_0}w_0
\ge \frac1{|Q_0|} \int_{Q_0\setminus B(\mathbf 0,2L_0)} w_0
\gtrsim w_{0,0}.
\end{equation*}
Since the tubes are pairwise disjoint, we deduce from
Lemmas \ref{lem:plateau-distance} and \ref{lem:tail} that
\begin{align*}
\fint_{Q_0} w_0^{-\frac{1}{p-1}}
&\ge \frac1{|Q_0|} \int_{\bigcup_{j=0}^\infty H_j} w_0^{-\frac{1}{p-1}}
\ge\frac{1}{|Q_0|}\sum_{j=0}^\infty |H_j|w_{0,j}^{-\frac{1}{p-1}}\\
&\sim L_0^{-n} \sum_{j=0}^\infty 2^j L_j^n w_{0,j}^{-\frac{1}{p-1}}
\sim w_{0,0}^{-\frac{1}{p-1}}\varepsilon^{-1}.
\end{align*}
Keeping only the $(0,0)$ entry of $L(x)L^{-1}(y)$ yields
\begin{equation*}
[W]_{\mathcal A_p}
\ge [w_0]_{A_p}
\ge \fint_{Q_0}w_0
\left(\fint_{Q_0}w_0^{-\frac{1}{p-1}}\right) ^{p-1}
\gtrsim \varepsilon^{-(p-1)}=\delta^{-1}.
\end{equation*}
This completes the proof of Proposition \ref{prop:matrix-weight}.
\end{proof}

\subsection{Proof of Theorem \ref{thm:main}}
\label{sec:test function}

In this subsection, we use the results obtained
in the preceding subsections to prove Theorem \ref{thm:main}.
To this end, we first establish the desired lower bound for
the operator norm of $T_\tau$,
whose off-diagonal kernel $K_\tau$ satisfies $K_\tau(-\cdot)=(-1)^qK_\tau(\cdot)$.
We continue to assume $p\in[2,\infty)$.
Recall that $q=1$ if $\tau=-$ and $q=2$ if $\tau=+$.
Let
$$
\mathcal Rf(\cdot):=f(-\cdot),\quad
T_+:=\frac{T+\mathcal RT\mathcal R}2,\quad\text{and}\quad
T_-:=\frac{T-\mathcal RT\mathcal R}2.
$$

\begin{proposition}\label{prop:norm-ratio}
Let $m\ge q+1$.
Then there exists a positive constant $C$ such that, for any $t\in[1,\infty)$,
\begin{equation} \label{lower}
\sup_{[W]_{\mathcal A_p}\le t}
\left\| T_\tau\right\|_{L^p(W)\to L^p(W)}
\ge C t^{1+\frac{1}{p(p-1)}}.
\end{equation}
\end{proposition}

\begin{proof}
We first consider the case $m=q+1$.
Given the parameters in \eqref{eq:params}, let
\begin{equation*}
N:=\lceil\delta^{-1}\rceil,\quad
M:=4N,\quad\text{and}\quad
J:=2N.
\end{equation*}
Choose indices $\{k_i\}_{i=0}^J$
as in Proposition \ref{non-degeneracy lower bound},
and construct the corresponding matrix weight $W$
as in Proposition \ref{prop:matrix-weight}.
Define the non-negative scalar function $u$ and the vector-valued test
function $\vec f$, respectively, by setting, for any $x\in\mathbb R^n$,
\begin{equation*}
u(x):=\eta\sum_{j=0}^{M-1}A_j^{-q}\mathbf1_{H_j}(x)
\quad\text{and}\quad
\vec f(x):=e^{\Phi(x)\mathfrak J_q}e_0 u(x).
\end{equation*}
Since the tubes are pairwise disjoint and only finitely many generations are involved,
we find that $u$ and $\vec f$ are bounded and compactly supported.
Note that $L\vec f=De_0u$ and $|W^{\frac1p}\vec f|=|L\vec f|$.
Using these, Lemma \ref{lem:plateau-distance},
and the definition of $w_0$ [see \eqref{eq:outside}], we obtain
\begin{align} \label{input}
\|\vec f\|_{L^p(W)}^p
=\int_{\mathbb R^n}|L(x)\vec f(x)|^p\,dx
= \int_{\mathbb R^n} w_0(x) |u(x)|^p\,dx
=\eta^p\sum_{j=0}^{M-1}|H_j|w_{0,j}A_j^{-qp},
\end{align}
which, together with \eqref{eq:plateau-volume}
and \eqref{eq:key-identities}, further implies that
\begin{align} \label{eq:input-size}
\|\vec f\|_{L^p(W)}^p
\sim \eta^p\sum_{j=0}^{M-1}V_jw_{0,j}A_j^{-qp}
=\eta^p\sum_{j=0}^{M-1}e^{-\delta j}
\lesssim\frac{\eta^p}{\delta}
\le\varepsilon.
\end{align}

Next, we estimate $\|T_\tau\vec f\|_{L^p(W)}$. Since
\begin{equation*}
\|T_\tau\vec f\|_{L^p(W)}^p
=\int_{\mathbb R^n} \left|W^{\frac1p}(x)T_\tau\vec f(x)\right|^p\,dx
=\int_{\mathbb R^n} \left|L(x)T_\tau\vec f(x)\right|^p\,dx,
\end{equation*}
it follows that taking the last coordinate of
$L T_\tau\vec f=De^{-\Phi\mathfrak J_q}T_\tau\vec f$ gives
\begin{equation}\label{eq:last-coordinate}
\|T_\tau\vec f\|_{L^p(W)}^p
\ge\int_{\mathbb R^n} w_q(x)
\left|\left[e^{-\Phi(x)\mathfrak J_q}T_\tau\vec f(x)\right]_q\right|^p\,dx.
\end{equation}

Fix a tube $H_I$ and denote the constant value of $\Phi$ on it
by $c_I$.
The scalar function $(\Phi-c_I)^q\frac{u}{q!}$ is bounded, compactly supported,
and vanishes on the open set $H_I$. Hence, inside $H_I$, the operator
$T_\tau$ applied to it is represented by the off-diagonal kernel with an
absolutely convergent integral.
Thus, for almost every $x\in H_I$ we have
\begin{align*}
\left[e^{-\Phi(x)\mathfrak J_q}T_\tau\vec f(x)\right]_q
&= \left[ T_\tau\left( e^{[\Phi(\cdot)-\Phi(x)]\mathfrak J_q}e_0 u(\cdot) \right) (x)\right]_q \\
&=\frac1{q!}\int_{\mathbb R^n}K_\tau(x-y)
[\Phi(y)-\Phi(x)]^q u(y)\,dy.
\end{align*}
For almost every $x\in H_I$, let
\begin{align*}
G(x)
&:=q!\operatorname{Re}\left\{(-1)^q\lambda
\left[e^{-\Phi(x)\mathfrak J_q}T_\tau\vec f(x)\right]_q\right\} \\
&\phantom{:}=\int_{\mathbb R^n}\operatorname{Re}\{\lambda K_\tau(x-y)
[\Phi(x)-\Phi(y)]^q\}u(y)\,dy,
\end{align*}
where the phase $\lambda$ is as in Proposition \ref{non-degeneracy lower bound}. Then
\begin{equation*}
\left|\left[e^{-\Phi(x)\mathfrak J_q}T_\tau\vec f(x)\right]_q\right|
\geq \frac1{q!} G(x),
\end{equation*}
and it remains to bound $G$ from below.

To this end, we split the domain of integration
$\{y\in\mathbb R^n:\ u(y)>0\}$ into four parts.
By the definition of $u$,
$u(y)>0$ if and only if $y\in \bigcup_{j=0}^{M-1} H_j$.
Thus, it suffices to consider the tubes $H_I$ contained in $\bigcup_{j=0}^{M-1} H_j$.
Fix $I\in\mathcal I_\ell$ and $x\in H_I$, where
$N\le\ell\le2N-1$. For $0\le i<\ell$, let $P_i\in\mathcal I_i$
be the ancestor of $I$ of generation $i$, and let
$P_i^{\mathrm{sib}}$ be the child of $P_i$ that does not contain $I$. Define
\begin{align*}
\mathcal B_i(I):=\bigcup_{j=i+1}^{M-1}
\bigcup_{\genfrac{}{}{0pt}{}{Q\in\mathcal I_j}{Q\subset P_i^{\mathrm{sib}}}} H_Q,
\quad
\mathcal A(I):=\bigcup_{i=0}^{\ell-1}H_{P_i},
\quad\text{and}\quad
\mathcal D(I):=\bigcup_{j=\ell+1}^{M-1}
\bigcup_{\genfrac{}{}{0pt}{}{Q\in\mathcal I_j}{Q\subset I}} H_Q.
\end{align*}
\begin{figure}[H]
\centering
\begin{tikzpicture}[x=1cm,y=1cm,font=\small]
\coordinate (r) at (6,0);
\coordinate (s0) at (3,-1.2);
\coordinate (p1) at (9,-1.2);
\coordinate (s00) at (1.5,-2.4);
\coordinate (s01) at (4.5,-2.4);
\coordinate (s1) at (7.5,-2.4);
\coordinate (cur) at (10.5,-2.4);

\foreach \a/\b in {r/s0,r/p1,s0/s00,s0/s01,p1/s1,p1/cur}
\draw[gray!65,thick] (\a) -- (\b);
\draw[gray!65,thick] (r) -- (p1) -- (cur);
\foreach \p/\a/\b in {s00/0.75/2.25,s01/3.75/5.25,
                 s1/6.75/8.25,cur/9.75/11.25}{
\draw[gray!65,thick] (\p) -- (\a,-3.6);
\draw[gray!65,thick] (\p) -- (\b,-3.6);
}

\node[right=4pt,text=myred] at (cur) {$H_I\, {\color{black}\ni x}$};
  \node[circle,inner sep=0pt,minimum size=5pt,fill=myred] at (cur) {};

\node[above=4pt] at (r) {$H_{P_0}$};
\node[right=4pt] at (p1) {$H_{P_1}$};
\foreach \p in {r,p1}
\node[circle,inner sep=0pt,minimum size=5pt,fill=black] at (\p) {};
\node[above=20pt] at (p1) {$\mathcal A(I)=H_{P_0}\cup H_{P_1}$};

\draw[draw=black,line width=1.5pt,dashed,rounded corners=5pt]
(9.15,-4.35) rectangle (11.85,-3.2);
\node[text=mygreen] at (10.5,-4.05) {$\mathcal D(I)$};
\foreach \a in {9.75,11.25}
\node[circle,inner sep=0pt,minimum size=5pt,fill=mygreen] at (\a,-3.6) {};

\draw[draw=black,line width=1.5pt,dashed,rounded corners=5pt]
(0.2,-4.35) rectangle (5.8,-0.8);
\draw[draw=black,line width=1.5pt,dashed,rounded corners=5pt]
(6.15,-4.35) rectangle (8.85,-2.0);

\foreach \p in {s00,s0,s01}
  \node[circle,inner sep=0pt,minimum size=5pt,fill=myblue] at (\p) {};

\node[circle,inner sep=0pt,minimum size=5pt,fill=myyellow] at (s1) {};

\foreach \a in {0.75,2.25,3.75,5.25}
  \node[circle,inner sep=0pt,minimum size=5pt,fill=myblue] at (\a,-3.6) {};

\foreach \a in {6.75,8.25}
  \node[circle,inner sep=0pt,minimum size=5pt,fill=myyellow] at (\a,-3.6) {};

\node[left=4pt,text=myblue] at (s0) {$H_{P_0^{\mathrm{sib}}}$};
\node[left=4pt,text=myyellow] at (s1) {$H_{P_1^{\mathrm{sib}}}$};

\node[text=myblue] at (3,-4.05) {$\mathcal B_0(I)$};
\node[text=myyellow] at (7.5,-4.05) {$\mathcal B_1(I)$};

\foreach \j in {0,...,3}
\node[left] at (-0.15,{-1.2*\j}) {$j=\j$};
\end{tikzpicture}
\caption{The four classes for fixed $x\in H_I$.
Every vertex represents one tube;
edges record interval ancestry, not inclusion of tubes.}
\end{figure}

Every tube belongs to exactly one of the following: a
sibling branch $\mathcal B_i(I)$, an ancestral tube in $\mathcal A(I)$,
the tube $H_I$ itself, or a descendant branch in $\mathcal D(I)$. This
gives the disjoint decomposition
\begin{equation*}
\{y\in\mathbb R^n:\ u(y)>0\}
=\left[\bigcup_{i=0}^{\ell-1}\mathcal B_i(I)\right]
\cup\mathcal A(I)\cup\mathcal D(I)\cup H_I.
\end{equation*}

We first treat the sibling branches.
Let $y\in\mathcal B_i(I)$. Let $\sigma:=1$ if the projection of $x$
lies in the right child of $P_i$, and $\sigma:=-1$ otherwise.
The two children abut the endpoints of $P_i$,
and $|x-(x,\vartheta)\vartheta|,
|y-(y,\vartheta)\vartheta|\le\frac{L_{i+1}}{10}$; hence
\begin{align*}
|(x-y)-\sigma L_i\vartheta|
&\leq |x-(x,\vartheta)\vartheta|
+|y-(y,\vartheta)\vartheta|
+|[(x,\vartheta)-(y,\vartheta)-\sigma L_i]\vartheta| \\
&\le \frac{L_{i+1}}{5}+2L_{i+1}
\le \frac{11}{5}\kappa^*L_i
< 4 \kappa^* L_i.
\end{align*}
Since $4\kappa^*\le\frac14$, the H\"older condition
\eqref{eq:holder} of $K_\tau$ implies
\begin{equation} \label{holder}
\left|K_\tau(x-y)-K_\tau(\sigma L_i\vartheta)\right|
\le H_K \frac{(4 \kappa^* L_i)^\alpha}{L_i^{n+\alpha}}
= (4 \kappa^*)^\alpha H_K L_i^{-n}.
\end{equation}

Since $\Phi$ is monotone and its oscillation on each child is
$A_{i+1}$, it follows that
\begin{equation} \label{one}
\sigma [\Phi(x)-\Phi(y)]
\ge A_i-2A_{i+1}
\ge\gamma A_i,
\end{equation}
where $\gamma:=1-2(\frac{e^{\frac1{10}}}8)^{\frac12}$
and the last inequality follows from \eqref{eq:A-decay}.
On the other hand, using $K_\tau(\sigma z)=\sigma^qK_\tau(z)$,
$|\lambda|=1$, and \eqref{holder}, we obtain
\begin{align*}
\sigma^q\operatorname{Re}\{\lambda K_\tau(x-y)\}
&=\operatorname{Re}\{\lambda\sigma^q K_\tau(x-y)\}\\
&=\operatorname{Re}\{\lambda K_\tau(L_i\vartheta)\}
+\operatorname{Re}\{\lambda\sigma^q
[K_\tau(x-y)-K_\tau(\sigma L_i\vartheta)]\}\\
&\ge \operatorname{Re}\{\lambda K_\tau(L_i\vartheta)\}
- | K_\tau(x-y)-K_\tau(\sigma L_i\vartheta)|\\
&\ge \operatorname{Re}\{\lambda K_\tau(L_i\vartheta)\}
- (4 \kappa^*)^\alpha H_K L_i^{-n}.
\end{align*}
This, together with \eqref{eq:sample} and \eqref{eq:A-choice}, further implies that
\begin{equation*}
\sigma^q\operatorname{Re}\{\lambda K_\tau(x-y)\}
\ge \left[a_0-(4 \kappa^*)^\alpha H_K\right]L_i^{-n}
\ge \frac{a_0}{2} L_i^{-n}.
\end{equation*}
From this and \eqref{one}, we deduce that,
for any $0\le i<\ell$ and $y\in\mathcal B_i(I)$,
\begin{equation}\label{eq:good-pair}
\operatorname{Re}\{\lambda K_\tau(x-y)
[\Phi(x)-\Phi(y)]^q\}
=\sigma^q\operatorname{Re}\{\lambda K_\tau(x-y)\}
\sigma^q [\Phi(x)-\Phi(y)]^q
\ge\frac{a_0\gamma^q}{2}L_i^{-n}A_i^q.
\end{equation}
In particular, every sibling branch gives a non-negative contribution.

For $0\le i<N$ and $i+1\le j\le M-1$, exactly
$2^{j-i-1}$ tubes of generation $j$ belong to
$\mathcal B_i(I)$. By \eqref{eq:plateau-volume} and
$A_j^q=2^jL_j^ne^{\eta j}$,
\begin{align*}
\int_{\mathcal B_i(I)}u(y)\,dy
&\gtrsim\eta\sum_{j=i+1}^{M-1}2^{j-i-1}L_j^nA_j^{-q}
=\eta2^{-i-1}\sum_{j=i+1}^{M-1}e^{-\eta j} \\
&=2^{-i-1}e^{-\eta i}
\frac{\eta e^{-\eta}}{1-e^{-\eta}}
\left[1-e^{-\eta(M-i-1)}\right]
\gtrsim 2^{-i}e^{-\eta i},
\end{align*}
where the last inequality follows from
$1-e^{-\eta}\sim \eta$ for all $0<\eta\le\frac1{10}$ and
\begin{equation*}
\eta(M-i-1)\ge3\eta N\ge3\delta N\ge3.
\end{equation*}
Summing \eqref{eq:good-pair} over the sibling branches therefore gives
the total positive contribution
\begin{align*}
&\int_{\bigcup_{i=0}^{\ell-1}\mathcal B_i(I)}
\operatorname{Re}\{\lambda K_\tau(x-y)
[\Phi(x)-\Phi(y)]^q\}u(y)\,dy \\
&\quad\gtrsim \sum_{i=0}^{N-1}
\int_{\mathcal B_i(I)}
L_i^{-n} A_i^q u(y)\,dy
\gtrsim\sum_{i=0}^{N-1}L_i^{-n}A_i^q2^{-i}e^{-\eta i}
=N.
\end{align*}

If $y\in H_{P_i}$, then $y$ lies above the middle of $P_i$ while $x$
lies above one of its children, so
\begin{equation*}
|x-y|
\ge \frac38 (L_i-2L_{i+1})
\ge \frac38 (L_i-2\kappa^*L_i)
\ge \frac{21}{64}L_i,
\end{equation*}
where the last inequality used $\kappa^*\le\frac1{16}$.
The size condition \eqref{eq:size} then gives
\begin{equation*}
|K_\tau(x-y)|
\le C_K \frac{1}{|x-y|^n}
\leq \left(\frac{64}{21}\right)^n C_K L_i^{-n}.
\end{equation*}
By this and \eqref{eq:Phi-interval-osc}, we conclude that
\begin{align*}
&\left|\int_{\mathcal A(I)}
\operatorname{Re}\{\lambda K_\tau(x-y)
[\Phi(x)-\Phi(y)]^q\}u(y)\,dy\right| \\
&\quad\lesssim \sum_{i=0}^{\ell-1}
\int_{H_{P_i}} L_i^{-n} A_i^q u(y)\,dy
= \sum_{i=0}^{\ell-1} |H_{P_i}| L_i^{-n} \eta
\sim \ell \eta.
\end{align*}

If $y\in\mathcal D(I)$, then similarly
we have $|K_\tau(x-y)|\le(\frac{64}{21})^nC_KL_\ell^{-n}$.
Moreover,
\begin{align*}
\int_{\mathcal D(I)}u(y)\,dy
&\sim \eta\sum_{j=\ell+1}^{M-1}2^{j-\ell}L_j^nA_j^{-q}
\leq \eta 2^{-\ell} \sum_{j=\ell+1}^\infty e^{-\eta j} \\
&= \eta 2^{-\ell} \frac{e^{-\eta (\ell+1)}}{1-e^{-\eta}}
\sim 2^{-\ell}e^{-\eta\ell}.
\end{align*}
Using these and \eqref{eq:Phi-interval-osc}, we obtain
\begin{equation*}
\left|\int_{\mathcal D(I)}
\operatorname{Re}\{\lambda K_\tau(x-y)
[\Phi(x)-\Phi(y)]^q\}u(y)\,dy\right|
\lesssim L_\ell^{-n}A_\ell^q2^{-\ell}e^{-\eta\ell}
=1.
\end{equation*}

Finally, for $y\in H_I$ we have $\Phi(y)=\Phi(x)$, so this part of the
integral vanishes.

Collecting the estimates for the four parts,
there exist positive constants $c_0,C_0$ such that
\begin{equation*}
G(x)\ge c_0N-C_0\eta\ell-C_0.
\end{equation*}
We now fix
\begin{equation*}
0<\varepsilon\le\varepsilon_0
:=\min\left\{\frac1{10},\frac{c_0}{8C_0},
\left(\frac{c_0}{4C_0}\right)^{\frac1{p-1}}\right\}.
\end{equation*}
This, together with $\ell<2N$, $\eta\le\varepsilon$,
and $N\ge\delta^{-1}=\varepsilon^{1-p}$, further implies that
$C_0\eta\ell\le 2C_0\varepsilon N\le \frac{c_0N}4$
and $C_0\le \frac{c_0N}4$.
Thus, for almost every $x\in \bigcup_{\ell=N}^{2N-1}H_\ell$,
\begin{equation*}
\left|\left[e^{-\Phi(x)\mathfrak J_q}T_\tau\vec f(x)\right]_q\right|
\ge\frac1{q!}G(x)\ge\frac{c_0}{2q!}N.
\end{equation*}
Inserting this into \eqref{eq:last-coordinate} gives
\begin{align*}
\|T_\tau\vec f\|_{L^p(W)}^p
&\gtrsim N^p \int_{\bigcup_{\ell=N}^{2N-1}H_\ell} w_q(x) \,dx
= N^p \sum_{\ell=N}^{2N-1} |H_\ell| w_{q,\ell} \\
&\sim N^p \delta \sum_{\ell=N}^{2N-1}e^{-\delta\ell}
\sim N^p \sim \delta^{-p},
\end{align*}
where the last two equivalences follow from $\delta N\sim 1$.
Combining the estimate of $\|T_\tau\vec f\|_{L^p(W)}$ and \eqref{eq:input-size}, we obtain
\begin{equation*}
\|T_\tau\|_{L^p(W)\to L^p(W)}
\ge\frac{\|T_\tau\vec f\|_{L^p(W)}}{\|\vec f\|_{L^p(W)}}
\gtrsim\delta^{-1}\varepsilon^{-\frac1p}
=\bigl(\delta^{-1}\bigr)^{1+\frac1{p(p-1)}},
\end{equation*}
where $[W]_{\mathcal A_p}\sim \delta^{-1}$ [see \eqref{eq:characteristics}].

The extension of the above result to \eqref{lower}
and $m\ge q+1$ is straightforward; see \cite[Theorem 3.13]{byyz}.
This completes the proof of Proposition \ref{prop:norm-ratio}.
\end{proof}

Finally, we prove Theorem \ref{thm:main}.

\begin{proof}[Proof of Theorem \ref{thm:main}]
We prove the present theorem by considering two cases for $p$.

\emph{Case (1)} $p\in[2,\infty)$.
In this case, by Proposition \ref{prop:norm-ratio}, we conclude that,
for any $t\in[1,\infty)$,
\begin{equation*}
\sup_{[W]_{\mathcal A_p}\le t}
\| T_\tau\|_{L^p(W)\to L^p(W)}
\gtrsim t^{1+\frac{1}{p(p-1)}},
\end{equation*}
where one may take any $m\ge2$ when $\tau=-$ and any $m\ge3$ when
$\tau=+$.
Observe that, for any $W\in \mathcal A_p$,
\begin{equation*}
\| T_\tau\|_{L^p(W)\to L^p(W)}
\le\frac12\left[
\| T\|_{L^p(W)\to L^p(W)}
+\| T\|_{L^p(\widetilde{W})\to L^p(\widetilde{W})}\right],
\end{equation*}
where $\widetilde W(\cdot):=W(-\cdot)$.
These, together with $[\widetilde W]_{\mathcal A_p}=[W]_{\mathcal A_p}$,
further imply that, for any $t\in[1,\infty)$,
\begin{equation*}
\sup_{[W]_{\mathcal A_p}\le t}
\left\| T\right\|_{L^p(W)\to L^p(W)}
\gtrsim t^{1+\frac{1}{p(p-1)}}.
\end{equation*}
When the odd part is itself non-degenerate, all colors have $\tau=-$
(see Remark \ref{key}), and the conclusion holds for all $m\ge2$;
in the general case it holds for all $m\ge3$.

\emph{Case (2)} $p\in(1,2)$.
In this case, $p'\in(2,\infty)$.
We apply Case (1) to the adjoint operator $T^*$.
The adjoint kernel $K^*(\cdot):=\overline{K(-\cdot)}$ satisfies
\eqref{eq:size},
\eqref{eq:holder}, and \eqref{eq:nd}; moreover,
$K^*_-(\cdot)=-\overline{K_-(\cdot)}$,
so the non-degeneracy of the odd part is preserved as well.
It is well known that $V\in\mathcal A_{p'}$
if and only if $\widetilde{V}:=V^{1-p}\in \mathcal A_p$ and
there exists a constant $C\in[1,\infty)$ such that
$$
C^{-1} [V]_{\mathcal A_{p'}}^{p-1}
\leq [\widetilde{V}]_{\mathcal A_p}
\leq C [V]_{\mathcal A_{p'}}^{p-1};
$$
see, for instance, \cite[Lemma A.6]{byy}.
By Case (1), we conclude that,
for $t\in[C,\infty)$, there exists $V\in\mathcal A_{p'}$ with
$[V]_{\mathcal A_{p'}}\le (\frac tC)^{\frac1{p-1}}$ such that
\begin{equation*}
\left\| T^*\right\|_{L^{p'}(V)\to L^{p'}(V)}
\gtrsim \left[ \left( \frac tC\right) ^{\frac1{p-1}} \right]^{1+\frac{1}{p'(p'-1)}}
\sim t^{1+\frac{1}{p(p-1)}}.
\end{equation*}
Then $[\widetilde{V}]_{\mathcal A_p}\le t$, and hence
\begin{equation*}
\sup_{[W]_{\mathcal A_p}\le t} \| T\|_{L^p(W)\to L^p(W)}
\ge \| T\|_{L^p(\widetilde{V})\to L^p(\widetilde{V})}
=\| T^*\|_{L^{p'}(V)\to L^{p'}(V)}
\gtrsim t^{1+\frac{1}{p(p-1)}}.
\end{equation*}
For any $t\in[1,C)$,
\begin{equation*}
\sup_{[W]_{\mathcal A_p}\le t} \| T\|_{L^p(W)\to L^p(W)}
\geq \| T\|_{L^p(I_m)\to L^p(I_m)}
\geq \| T\|_{L^p(I_m)\to L^p(I_m)}
C^{-[1+\frac{1}{p(p-1)}]} t^{1+\frac{1}{p(p-1)}}.
\end{equation*}
This completes the proof of Theorem \ref{thm:main}.
\end{proof}

\begin{proof}[Proof of Corollary \ref{cor:homog}]
The kernel $\frac{1}{\pi z}$ of the Hilbert transform on $\mathbb R$
and the kernel $\frac{c_nz_j}{|z|^{n+1}}$ of the $j$-th Riesz transform
on $\mathbb R^n$ ($j\in\{1,\ldots,n\}$ and $c_n$ is a nonzero normalizing constant)
are smooth odd kernels away from the origin, for which
\eqref{eq:size} and \eqref{eq:holder} are immediate; along the
direction $z=t>0$ in dimension one (or $z=te_j$ with $t>0$ in the Riesz case), the
kernel is a nonzero constant multiple of $t^{-n}$, so the odd part
itself satisfies \eqref{eq:nd}.
Thus, Theorem \ref{thm:main} applies with $m=2$ to these operators.
This completes the proof of Corollary \ref{cor:homog}.
\end{proof}

\noindent\textbf{Acknowledgements}\quad
The authors acknowledge the use of AI tools during the exploratory stage of this project.
All mathematical arguments and proofs in the final manuscript
were checked and written by the authors.

\bigskip

\noindent Fan Bu

\medskip

\noindent Department of Mathematics, Faculty of Arts and Sciences,
Beijing Normal University, Zhuhai 519087, The People's Republic of China

\smallskip

\noindent{\it E-mail:} \texttt{fanbu@bnu.edu.cn}

\bigskip

\noindent Dachun Yang (Corresponding author) and Wen Yuan

\medskip

\noindent Laboratory of Mathematics and Complex Systems (Ministry of Education of China),
School of Mathematical Sciences,  Institute for Advanced Study,
Beijing Normal University, Beijing 100875, The People's Republic of China

\smallskip

\noindent{\it E-mails:} \texttt{dcyang@bnu.edu.cn} (D. Yang)

\noindent\phantom{{\it E-mails:}} \texttt{wenyuan@bnu.edu.cn} (W. Yuan)

\end{document}